\documentclass[3p,number,sort]{elsarticle}

\makeatletter
\def\ps@pprintTitle{%
 \let\@oddhead\@empty
 \let\@evenhead\@empty
 \def\@oddfoot{}%
 \let\@evenfoot\@oddfoot}
\makeatother

\newtheorem{theorem}{Theorem}
\newtheorem{lemma}[theorem]{Lemma}
\newtheorem{proposition}[theorem]{Proposition}
\newtheorem{corollary}[theorem]{Corollary}

\newdefinition{definition}[theorem]{Definition}
\newdefinition{conjecture}[theorem]{Conjecture}
\newdefinition{example}[theorem]{Example}
\newdefinition{remark}[theorem]{Remark}
\newdefinition{note}[theorem]{Note}
\newdefinition{case}[theorem]{Case}

\newproof{proof}{proof}

\usepackage{amsmath,amssymb}
\usepackage{bbm}
\usepackage{enumitem}

\usepackage{setspace} 

\usepackage{hyperref}
\usepackage{cleveref}

\usepackage{float}

\usepackage{pgfplots}
\pgfplotsset{compat=newest}

\usepgfplotslibrary{patchplots}

\usepackage{tikz}
\usetikzlibrary{math}
\usetikzlibrary{external}
\usepackage{ulem}
\usepackage{multirow}

\allowdisplaybreaks[1]
\usepackage[misc]{ifsym}
\newcommand{\envelope}{(\kern1pt\Letter\kern1pt)}

\newcommand{\N}{\ensuremath{\mathbb{N}}}

\newcommand{\R}{\ensuremath{\mathbb{R}}}
\newcommand{\C}{\ensuremath{\mathbb{C}}}
\newcommand{\e}{\ensuremath{\varepsilon}}
\newcommand{\p}{\ensuremath{\varphi}}

\newcommand{\dt}[2]{\mbox{#1.\hspace{.13em}#2.}}
\newcommand{\dtt}[3]{\mbox{#1.\hspace{.13em}#2.\hspace{.13em}#3.}}
\newcommand{\dtf}[4]{\mbox{#1.\hspace{.13em}#2.\hspace{.13em}#3.\hspace{.13em}#4.}}

\newcommand*\cvec[1]{\begin{pmatrix}#1\end{pmatrix}}

\newcommand{\sgn}{\mathrm{sgn}}
\newcommand{\Langle}{\left\langle}
\newcommand{\Rangle}{\right\rangle}
\newcommand{\T}{\mathcal{T}}

\newcommand{\Id}{\mathrm{Id}}
\newcommand{\lcm}{\mathrm{lcm}}

\newcommand{\1}{\mathbbm{1}}

\renewcommand{\S}{\mathcal{S}}
\renewcommand{\O}{\mathcal{O}}

\begin{document}

\begin{frontmatter}

\title{A space-based method for the generation of a Schwartz function with infinitely many generalized vanishing moments with applications in image processing}

\author[rvt]{T. Fink \corref{cor1}}
\ead{thomas.fink@uni-passau.de}

\author[rvt]{U. K\"ahler}
\ead{ukaehler@ua.pt}

\cortext[cor1]{Corresponding author}

\address{Universit\"at Passau, Universidade de Aveiro}

\begin{abstract}
In this article we construct a function with infinitely many vanishing (generalized) moments. This is motivated by an application to the Taylorlet transform which is based
on the continuous shearlet transform. It can detect curvature and other higher order geometric information of singularities in addition to their position and the direction. For a robust detection of these features a function with higher order vanishing moments, $\int_\R g(x^k)x^mdx=0$, is needed. We show that the presented construction produces an explicit formula of a function with $\infty$ many vanishing moments of arbitrary order and thus allows for a robust detection of certain geometric features. The construction has an inherent connection to q-calculus, the Euler function and the partition function.
\end{abstract}

\end{frontmatter}

\section{Introduction}

Vanishing moment conditions give orthogonality with respect to subspaces of polynomials and therefore play a vital role in many areas of analysis. Especially in wavelet and shearlet theory they are of pivotal importance.  A wavelet needs vanishing moments to enable a detection of the regularity of the analyzed signal \cite[Thm 2]{MaHw92}. Similarly, for the resolution of the wavefront set by the continuous shearlet transform, a shearlet has to incorporate so called vanishing directional moments \cite[Thm 6.1 $\&$ 6.4]{gr11}. For a shearlet $\psi\in L^2(\R^2)$, they are of the form $\int_\R\psi(x_1,x_2)x_1^m dx_1=0$ for all $x_2\in\R$.

While the continuous shearlet transform allows for a detection of the position and direction of a singularity, the recently created Taylorlet transform additionally allows for a detection of the curvature and other higher order geometric information of singularities. The Taylorlets inherit the properties of the shearlets and extend them by utilizing shears of higher order, \dt ie
\[S_s(x):=\begin{pmatrix}
x_1+\sum_{k=0}^n \frac{s_k}{k!}\cdot x_2^k \\
x_2
\end{pmatrix}\quad\text{for } x\in\R^2\text{ and }s=(s_0,\ldots,s_n)\in\R^{n+1}.\] 
The Taylorlet transform of a function $f\in L^2(\R^2)$ is defined as $L^2$-scalar product of $f$ and a dilated, translated and sheared version of a Taylorlet. This transform allows for a detection of certain geometric features of the singularities of the analyzed function by observing the transform's decay rate for decreasing scales. The decay rate depends on the choice of the translation and shear parameters and on the so called vanishing directional moments of higher order of a Taylorlet $\tau\in\S(\R^2)$, \dt ie conditions of the form $\int_\R\tau(\pm x_1^k,x_2)x_1^mdx_1=0$ for all $x_2\in\R$, where $k,m\in\N$ \cite{F17}. Thus, a function with infinitely many vanishing moments of higher order is both of great theoretical interest and very useful for a robust detection of geometric information.

For the construction of functions with vanishing moments there exist two classical approaches. The first method uses an arbitrary Schwartz function $f\in\S(\R)$, whose $n^{\text{th}}$ derivative has $n$ vanishing moments. The most famous example is probably the Mexican hat wavelet which is the second derivative of the Gaussian and thus exhibits 2 vanishing moments. The drawback of this approach is its limited use, if one is interested in a function with infinitely many vanishing moments. To this end, a Fourier ansatz is more convenient. Since the number of vanishing moments of a function coincides with the order of the Fourier transform's root in the origin, it suffices to construct a function in the Fourier domain which vanishes with a proper order in the origin. A well known example of this method is the Meyer wavelet, whose Fourier support is $\left[-\frac{8\pi}3,-\frac{2\pi}3\right]\cup\left[\frac{2\pi}3,\frac{8\pi}3\right]$. Hence, the Meyer wavelet exhibits infinitely many vanishing moments.

Yet, under certain circumstances an explicit formula for such a function in space domain is preferable over a Fourier construction. For instance, the higher order vanishing moment conditions $\int_\R g(x^k)x^mdx=0$ do not interact well with the Fourier transform due to their non-linear nature and hence the construction and explicit expression of a function incorporating these conditions is easier to achieve and to apply in space domain. Hence, we consider a construction yielding an explicit formula in space domain for a function with infinitely many vanishing moments.

The generation of vanishing moments is achieved by considering linear combinations of dilations of a function. This process yields a structure which can be studied by applying a q-calculus of operator-valued functions. This calculus is a variation of the classical analysis and resembles the finite difference calculus, but uses a multiplicative notation instead. \dt Eg, for $q>0$, the q-derivative of a function $f\in C(\R)$ in $x\in\R$ is defined as
\[d_q f(x) = \frac{f(qx)-f(x)}{qx-x}.\]
This calculus recently gained interest due to its applications in quantum mechanics. The construction we present in this article has an inherent connection to the q-Pochhammer symbol and the Euler function. The latter itself incorporates a deep link to the theory of partitions.

This article is structured as follows. In the second section, we define the Taylorlet transform and show its most important properties. Section 3 contains a small introduction into q-calculus. The fourth section is dedicated to the construction of Schwartz functions with infinitely many vanishing moments and highlights its connection to q-calculus. In section 5, we give a numerical analysis of the evaluation of the function constructed in the previous section and show numerical examples of its application to the Taylorlet transform. Finally, the last section incorporates a conclusion of the article and an outline of a possible generalization of the wavefront set.

\section{Taylorlets and higher order vanishing moments}

A classical result in the theory of shearlets which proves its value is the resolution of the wavefront set by the continuous shearlet transform which was shown by Kutyniok and Labate \cite{KuLa09}. Suppose, the function $f\in\S'\left(\R^2\right)$ analyzed with the continuous shearlet transform has a singular support that can be represented as graph of a function $q\in C^\infty(\R)$, \dt eg
\[f(x)=\1_{\R_+}\left(x_1-q(x_2)\right)\quad\text{for all }x\in\R^2.\]
We will call $q$ the singularity function. Under these circumstances, the wavefront set of $f$ can be interpreted as a linear approximation of $q$. In this scenario, the shearlet transform $\mathcal{SH}_\psi f(a,s,t)$ does not decay faster than any polynomial for $a\to 0$, if and only if $q(t_2)=t_1$ and $q'(t_2)=s$. In other words, the continuous shearlet transform provides the first two Taylor coefficients of $q$ at the point $t_2$. The Taylorlet transform expands this idea by supplying means to detect arbitrary Taylor coefficients of the singularity function $q$.

\begin{definition}[Vanishing moments of higher order, analyzing Taylorlet, restrictiveness]
We say that a function $f:\R\to\R$ has $M$ vanishing moments of order $n$ if
\[\int_\R f\big(\pm t^k\big)t^m dt = 0\]
for all $m\in\{0, \ldots, kM-1\}$ and for all $k\in\{1,\ldots,n\}.$

Let $\S(\R^{d})$ denote the Schwartz space on $\R^{d}$ and let $g,h\in\S(\R)$ such that $g$ has $M$ vanishing moments of order $n$. The space of Schwartz functions with infinitely many vanishing moments of order $n$ will be denoted as $\S^*_n(\R).$ We call the function
\[\tau=g\otimes h\]
an analyzing Taylorlet of order $n$ with $M$ vanishing moments in $x_1$-direction.

We say $\tau$ is restrictive, if additionally 
\begin{enumerate}[label=(\roman*)]
\item $g(0)\ne 0$ and $\int_0^\infty g(t)t^jdt\ne 0$ for all $j\in\{0,\ldots,M-1\}$ and
\item $\int_\R h(t)dt\ne 0$.
\end{enumerate}
\end{definition}

\begin{figure}[H]
\centering
\begin{tikzpicture}[scale=.6]
\begin{axis}[samples=500,domain=-6:6]
\addplot[ultra thick,blue!80!black]plot (\x, {1/630*exp(-pow(\x,2))*(1+\x)*(315 - 51660*pow(\x,2) + 286020*pow(\x,4) - 349440*pow(\x,6) + 142464*pow(\x,8) - 21504*pow(\x,10) + 1024*pow(\x,12))});
\end{axis}
\end{tikzpicture}
\caption{Function with 3 vanishing moments of order 2}
\end{figure}
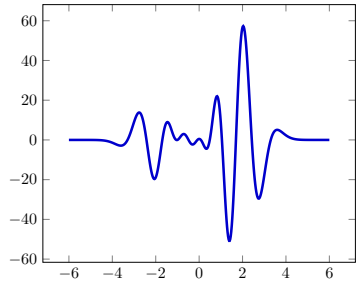

The Taylorlet transform is defined as follows.

\begin{definition}[Taylorlet transform]
Let $n\in\N$ and let $\tau\in\S(\R^2)$ be an analyzing Taylorlet of order $n$. Let $\alpha>0$, $t\in\R$, $a>0$ and $s\in\R^{n+1}$. We define
\[\tau^{(n,\alpha)}_{a,s,t}(x):= \tau\cvec{\left[x_1-\sum_{k=0}^n \frac {s_k}{k!}\cdot (x_2-t)^k\right]/a \\ (x_2-t)/a^\alpha}\quad\text{for all }x=(x_1,x_{2})\in\R^{2}.\]
The Taylorlet transform \dtt wrt $\tau$ of a tempered distribution $f\in \S'(\R^2)$ is defined as
\begin{align*}
\mathcal{T}^{(n,\alpha)}:\S'\left(\R^2\right)\to C^\infty\left(\R_+\times\R^{n+1}\times\R\right),\quad \mathcal{T}^{(n,\alpha)}f(a,s,t) = \Langle f, \tau^{(n,\alpha)}_{a,s,t}\Rangle.
\end{align*}
\end{definition}

\begin{table}[H]
\centering
\begin{tabular}{|c|c|c|}
\hline
Analyzing function & Moment condition & Detected geometric features \\ \hline
Shearlet & $\int_\R \psi(x_1,x_2)x_1^mdx_1=0$ & Position and direction \\ 
$\psi\in\S\left(\R^2\right)$ & for all $m\in\N$ & of singularities \\
\hline
Taylorlet & $\int_\R g(t)t^mdt=0=\int_\R g(\pm t^2)t^m dt$ & Position, direction and \\ 
$\tau=g\otimes h\in\S\left(\R^2\right)$ & for all $m\in\N$ & curvature of singularities \\
\hline
Taylorlet & $\int_\R g(\pm t^k)t^mdt=0$ & First n+1 Taylor cofficients \\ 
$\tau=g\otimes h\in\S\left(\R^2\right)$ & for all $k\in\{1,\ldots,n\}$, $m\in\N$ &  of the singularity function \\
\hline
\end{tabular}
\caption{Moment conditions and detection results}
\end{table}
\vskip5mm

The Taylorlet transform allows for the following detection result.
\vskip5mm

\begin{theorem} \label{detect}
Let $M,n\in\N$ and let $\tau$ be an analyzing Taylorlet of order $n$ with $M$ vanishing moments in $x_1$-direction. Let furthermore $0\le j< M-1$, $t\in\R$ and let $q\in C^\infty(\R)$ be the singularity function of
\[f(x)=\left[x_1-q(x_2)\right]^{j}\cdot \1_{\R_\pm}\left(x_1-q(x_2)\right).\]
\begin{enumerate}

\item Let $\alpha>0$. If $s_0\ne q(t)$, the Taylorlet transform has a decay of
\[\mathcal{T}^{(n,\alpha)}f(a,s,t)=\O\big(a^N\big)\quad\text{for }a\to 0\]
for all $N>0$.

\item Let $\alpha<\frac 1n$, $k<n$ and let $s_\ell=q^{(\ell)}(t)$ for all $\ell\in\{0,\ldots,k-1\}$. Then the Taylorlet transform has the decay property
\[\mathcal{T}^{(n,\alpha)}f(a,s,t)=\O\left(a^{j-1+(M-j-1)[1-(k+1)\alpha]}\right)\quad\text{for }a\to 0.\]

\item Let $\alpha>\frac 1{n+1}$ and let $\tau$ be restrictive. If $s_\ell=q^{(\ell)}(t)$ for all $\ell\in\{0,\ldots,n\}$, then the Taylorlet transform has the decay property
\[\mathcal{T}^{(n,\alpha)}f(a,s,t)\sim a^{j}\quad\text{for }a\to 0.\]
\end{enumerate}

\end{theorem}

Due to the detection result, the construction of a function $g\in \S_n^*(\R)$ is highly desirable, as the corresponding Taylorlet $\tau= g\otimes h$ allows for a very robust detection of the Taylor coefficients of the singularity function. Furthermore, such a Taylorlet simplifies said detection, as shown in the following corollary.

\begin{corollary}\label{cor}
Let $n\in\N$ and let $\tau$ be a restricitve analyzing Taylorlet of order $n$ with infinitely many vanishing moments in $x_1$-direction. Let furthermore $\alpha\in\left(\frac 1{n+1},\frac 1n\right)$, $j\ge 0$ and let $q\in C^\infty(\R)$ be the singularity function of
\[f(x)=\left[x_1-q(x_2)\right]^{j}\cdot \1_{\R_\pm}\left(x_1-q(x_2)\right).\]
Then
\[\mathcal{T}^{(n,\alpha)}f(a,s,t)=\O\big(a^N\big)\quad\text{for }a\to 0\]
for all $N>0$, if and only if there exists a $k\in\{0,\ldots,n\}$ such that $s_k\ne q^{(k)}(t)$.
\end{corollary}

\section{q-Calculus}

Despite its modern applications in quantum mechanics, the first accounts of q-calculus actually date back to the days of Euler. When he developed the theory of partitions, he introduced the partition function $p:\N\to\N$ with $p(n)$ being the number of distinct ways of representing $n$ as a sum of natural numbers. He found out that the infinite product
\[\prod_{k=1}^\infty \frac 1{1-q^k}=\sum_{n=0}^\infty p(n) q^n\]
is the generating function for the partition function \cite{Er00}. Its reciprocal is also known as Euler's function.
\begin{definition}[Euler's function]
Let $q\in\C$ with $|q|<1$. Then,
\[\p(q)=\prod_{k=1}^\infty(1-q^k)\]
is Euler's function.
\end{definition}
This function is not to be confused with Euler's totient function which is also denoted by $\p$, but $\p(n)$ there displays the amount of numbers up to $n$ which are relative prime to $n$.

The concept of q-calculus is similar to that of the finite difference calculus, but the q-derivative of a function $f:\R\to\R$ is defined as
\[d_qf(x)=\frac{f(qx)-f(x)}{qx-x}\]
rather than $D_hf(x)=\frac{f(x+h)-f(x)}h.$ It is of particular importance in structures based on q-commutator relations such as the Manin plane. Such structures appear not only in relation to quantum groups, but also in terms of interpolation between the bosonic and the fermionic case \cite{BS91}. For a more general overview on q-calculus see \cite{Kac}.
Similar to the finite differences, the infinitesimal theory can be obtained by q-calculus via the limit process $q\to 1$.

As an example, the $q$-derivative of a monomial can be found to be
\[d_q(x^n)=\frac{q^n-1}{q-1}\cdot x^{n-1}.\]
The occuring factor is also known as the q-bracket $[n]_q=\frac{q^n-1}{q-1}$. This also yields a generalization of the classical binomial coefficient
\[\binom nk_q=\prod_{k=1}^{n}\frac{[n+1-k]_q}{[k]_q}.\]
\begin{table}
\centering
\begin{tabular}{|c|c|}
\hline
Infinitesimal calculus & q-calculus \\ \hline
$f'(x)=\lim_{q\to 1}\frac {f(qx)-f(x)}{qx-x}$ & $d_qf(x)=\frac {f(qx)-f(x)}{qx-x}$ \\ \hline
$\frac d{dx}\ x^n= n\cdot x^{n-1}$ & $d_q (x^n) =[n]_q\cdot x^{n-1}=\frac{q^n-1}{q-1}\cdot x^{n-1}$ \\ \hline
$n!$ & $[n]_q!=\prod_{k=1}^n [k]_q$ \\\hline
$\binom nk$ & $\binom nk_q=\frac{[n]_q!}{[k]_q!\cdot[n-k]_q!}$ \\\hline
$(a)_n$ & $(a;q)_n=\prod_{k=0}^{n-1}(1-aq^k)$ \\\hline
\end{tabular}
\caption{Overview of important q-analogs}
\end{table}
One of the most central concepts in q-calculus is the analog of the classical Pochammer symbol. It is defined as
\[(a;q)_n:=\prod_{k=0}^{n-1}\left(1-aq^{k}\right).\]
It can be represented in terms of the q-binomial as shown in the following lemma.
\begin{lemma}\label{qbin0}
\cite[(4.2.3)]{Ex83}
Let $x\in\C$, $q>0$ and $n\in\N$. Then
\[(x;q)_n=\sum_{k=0}^n \binom nk_{q}q^{\binom k2}(-1)^k\cdot x^k.\]
\end{lemma}
The statement of this lemma can be extended from $x\in\C$ to automorphism on $\C$-vector spaces as shown in the next corollary. This will be important later in this article.
\begin{corollary}\label{qbin}
Let $q>0$, $n\in\N$ and let $(V,\|\cdot\|)$ be a normed vector space over $\C$. Furthermore, let $A\in\mathcal{L}(V,V)$, \dt ie, $A:V\to V$ is linear and continuous with respect to $\|\cdot\|$. Then
\[ (A;q)_n := \prod_{k=0}^{n-1}\left(\mathrm{Id}-q^{k}\cdot A\right)\in\mathcal{L}(V,V) \]
and
\[(A;q)_n = \sum_{k=0}^n \binom nk_{q}q^{\binom k2}(-1)^k\cdot A^k. \]
\end{corollary}
\begin{proof}
Since $A\in\mathcal{L}(V,V)$, we have $\mathrm{Id}-q^{k}\cdot A\in\mathcal{L}(V,V)$ for all $q\in\C$ and all $k\in\N$ and so
\[ \prod_{k=0}^{n-1}\left(\mathrm{Id}-q^{k}\cdot A\right) \in \mathcal{L}(V,V)\quad\text{for all }q\in\C,n\in\N. \]
Since $\{\mathrm{Id}-q^{k}\cdot A\}_{k\in\N}\cup\{A\}$ commute, the proof for the identity
\[(A;q)_n = \sum_{k=0}^n \binom nk_{q}q^{\binom k2}(-1)^k\cdot A^k \]
is analogous to the case of Lemma \ref{qbin0}.
\end{proof}

The q-Pochhammer symbol also is the initial point for a multitude of important functions in q-calculus. Among them, the Euler function can be represented as
\[\p(q)=(q;q)_\infty.\]
As we will see in \Cref{construction} in the next section, the Euler function can also be expanded as a series of q-Pochhammer symbols.

\section{Construction}

In order to obtain a function $g\in \S_n^*(\R)$, it is sufficient to construct a function $\psi$ such that
\begin{enumerate}
\item $\psi$ is even,
\item $\psi^{(k)}(0)=c\cdot\delta_{0k}$ for some $c\ne 0$,
\item $\psi\in\S_1^\ast(\R)$.
\end{enumerate}
As the following proposition shows, with such a function $\psi$, we can construct a function with infinitely many vanishing moments of arbitrary order $n$.

\begin{proposition}\label{sqrt}
Let the function $\psi\in\S(\R)$ fulfill the conditions 1. and 2. and let $v_{n}:=\lcm\{1,\ldots,n\}$ be the least common multiple of the numbers $1,\ldots,n$.
\begin{enumerate}
\item[a)] If $M\in\N$ and $\psi$ has $Mv_n$ vanishing moments, the function
\[g:=\psi\circ \sqrt[v_{n}]{|\cdot|}\in\S(\R)\]
and has $M$ vanishing moments of order $n$.
\item[b)] If $\psi$ fulfills condition 3., the function
\[g:=\psi\circ \sqrt[v_{n}]{|\cdot|}\in\S_n^*(\R).\]
\end{enumerate}
\end{proposition}
\begin{proof}
\begin{enumerate}[label=\textit{\alph*)}]
\item The decay conditions of $\psi$ are not changed by the concatenation with $\sqrt[v_{n}]{|\cdot|}$ and its smoothness is preserved as well due to condition 2. Hence, $g\in\S(\R)$. Hence, it only remains to prove, that $g$ has $M$ vanishing moments of order $n$.
\begin{align*}
\int_{\R} g\big(\pm t^k\big) t^m dt &= \int_{\R} \psi\big(|t|^{k/v_n}\big) t^m dt \\
&= \int_{\R} \psi(|u|) u^{m v_n/k}\cdot\frac{v_n}{k} u^{v_n/k-1}du \\
&= \frac{v_n}k\cdot \int_{\R} \psi(u) u^{(m+1) v_n/k-1} du=0
\end{align*}
for all $m\in\{0,\ldots,kM-1\}$, since $k$ divides $v_n$ for all $k\in\{1,\ldots,n\}$.
\item This follows immediately from a).

\end{enumerate}
\end{proof}

For the construction of a function $\psi$ with properties 1. - 3. we start with an even bump function $\phi\in C_{c}^{\infty}(\R)$ and a number $\e>0$ such that $\phi\big|_{\left[-\e,\e\right]}\equiv 1.$ Hence, properties 1. and 2. are already fulfilled, $\phi$ is a Schwartz function and we only need to gather vanishing moments. This can be achieved for $q>1$ by the following function sequence:
\[\phi_{m+1}=\left(\Id-q^{-(m+1)}D_{\frac 1q}\right)\phi_{m},\]
where for $a>0$, $D_a$ is the dilation operator with $D_a:L^\infty(\R)\to L^\infty(\R)$, $D_af(x)=f(ax)$. 

The innate connection between the function sequence $\phi_m$ and the q-derivative can be seen in the following proposition.
\begin{proposition}\label{Fourier}
Let $n\in\N$, $\phi_0\in L^1(\R,x^ndx)$ and
\[\phi_{m+1}=\left(\Id-q^{-(m+1)}\cdot D_\frac 1q\right)\phi_m\]
for all $m\in\N$. Then 
\[\widehat{\phi_n}(\omega)=(1-q)^n\cdot\omega^n\cdot d_q^n\widehat{\phi_0}(\omega).\]
\end{proposition}

\begin{proof}
We first look for an appropriate representation of $\phi_n$. To this end, we will utilize Corollary \ref{qbin}. Since the dilation operator is linear, we can write $\phi_n$ as an operator-valued q-Pochhammer symbol
\[\phi_n=\prod_{m=0}^{n-1} \left(\Id-q^{-(m+1)}D_{\frac 1q}\right)\phi_0 = \left(q^{-1} D_{\frac 1q}; q^{-1}\right)_n\phi_0\]
and obtain that
\begin{equation}\label{phi}
\phi_n=\sum_{k=0}^n \binom nk_{\frac 1q}q^{-\binom k2}(-1)^k\cdot q^{-k}D_{q^{-k}}\phi_0.
\end{equation}
Due to \cite[(6.98)]{Er12}, the $n^{\mathrm{th}}$ q-derivative of a function can be represented as
\[d_q^nf(x)=(q-1)^{-n}q^{-\binom n2}x^{-n}\sum_{k=0}^n\binom nk_q q^{\binom k2}(-1)^k f\left(q^{n-k}x\right).\]
Hence, we have to show that
\begin{equation}\label{goal}
\widehat{\phi_n}(\omega) = q^{-\binom n2}\sum_{k=0}^n\binom nk_q q^{\binom k2}(-1)^{n-k} \widehat{\phi_0}\left(q^{n-k}\omega\right).
\end{equation}
To this end, we first represent $\binom nk_{\frac 1q}$ in terms of $\binom nk_q$. We can write
\begin{align*}
\binom nk_{\frac 1q} &= \prod_{\ell=1}^k \frac{1-q^{\ell-n-1}}{1-q^{-\ell}} \\
&= \prod_{\ell=1}^k \frac{q^{\ell-n-1}}{q^{-\ell}}\cdot \frac{q^{n+1-\ell}-1}{q^{\ell}-1} \\
&= q^{2\binom{k+1}2-k(n+1)}\cdot \prod_{\ell=1}^k \frac{q^{n+1-\ell}-1}{q^{\ell}-1} \\
&= q^{-k(n-k)}\cdot \binom nk_q.
\end{align*}
By applying the Fourier transform to equation (\ref{phi}) and inserting the upper equality, we get
\begin{align*}
\widehat{\phi_n}(\omega) &= \sum_{k=0}^n \binom nk_{\frac 1q}q^{-\binom k2}(-1)^k\widehat{\phi_0}\left(q^k \omega \right) \\
&= \sum_{k=0}^n \binom nk_{\frac 1q}q^{-\binom {n-k}2}(-1)^{n-k}\widehat{\phi_0}\left(q^{n-k}\omega\right) \\
&= \sum_{k=0}^n \binom nk_q q^{-k(n-k)} q^{-\binom {n-k}2}(-1)^{n-k}\widehat{\phi_0}\left(q^{n-k}\omega\right) \\
&= q^{-\binom n2}\cdot\sum_{k=0}^n \binom nk_q q^{\binom k2} (-1)^{n-k}\widehat{\phi_0}\left(\textcolor{black}{q^{n-k}\omega}\right),
\end{align*}
where the last equality results from $\binom {a+b}2=\binom a2+ab+\binom b2$. 
\end{proof}

As a consequence of this proposition, each step of the recursion \[\phi_{m+1}=(\Id-q^{-(m+1)}D_{1/q})\phi_{m}\] generates one further vanishing moment. As the next lemma shows, not only $\phi_m$, but also its restriction to $\R_+$ or $\R_-$ features $m$ vanishing moments.

\begin{lemma}\label{vanmom}
Let $\phi_0\in\S(\R)$ be an even function and let
\[\phi_{m+1}=\left(\Id-q^{-(m+1)}D_{\frac 1q}\right)\phi_{m}\]
for all $m\in\N$. Then $\phi_m\in\S(\R)$ and
\[\int_{\R_\pm}\phi_m(x) x^\ell dx=0\]
for all $\ell\in\{0,\ldots,m-1\}$ and for all $m\in\N$.
\end{lemma}
\begin{proof}
The statement can be shown inductively. Obviously, the statement is true for $\phi_0$. Now we assume, that $\phi_m\in\S(\R)$ and $\int_{\R_\pm}\phi_m(x) x^\ell dx=0$ for all $\ell\in\{0,\ldots,m-1\}$. It can be easily confirmed that $\phi_{m+1}\in\S(\R)$. Furthermore, for all $\ell\in\{0,\ldots,m-1\}$, we have
\[\int_{\R_\pm}\phi_{m+1}(x)x^\ell dx=\underbrace{\int_{\R_\pm}\phi_{m}(x)x^\ell dx}_{=0} - q^{-(m+1)}\cdot\underbrace{\int_{\R_\pm}\phi_m\left(\frac xq\right) x^\ell dx}_{=0}.\]
Furthermore,
\begin{align*}
\int_{\R_\pm}\phi_{m+1}(x)x^mdx &= \int_{\R_\pm} \left[\phi_m(x)-q^{-m-1}\phi_m\left(\frac xq\right)\right]x^m dx \\
&= \int_{\R_\pm} \phi_m(x)x^m dx - \int_{\R_\pm}\phi_m\left(\frac xq\right) \cdot\left(\frac xq\right)^m \frac{dx}q \\
&= \int_{\R_\pm} \phi_m(x)x^m dx - \int_{\R_\pm} \phi_m(x)x^m dx \\
&= 0.
\end{align*}
\end{proof}

Since $\int_0^\infty \phi_m(x)dx =0$ for all $m\ge 1$ as a consequence of the previous lemma, we cannot immediately use the functions $\phi_m$ to produce restrictive Taylorlets. We will present a method to achieve this in Lemma \ref{rest}. The next lemma shows that the sequence $\phi_m$ converges to a function satisfying the properties 1. - 3. 

\begin{lemma}\label{existence}
Let $\phi_{0}\in C_{c}^{\infty}(\R)$ and let $q>1$ and $\e>0$ such that $\phi_{0}\big|_{\left[-\e,\e\right]}\equiv 1.$ and let 
\[\phi_{m+1}=\left(\mathrm{Id}-q^{-(m+1)}D_{\frac 1q}\right)\phi_{m}\]
 for all $m\in\N$. Then the function sequence $\phi_m$ converges uniformly to $\psi$ for $m\to\infty$ and $\psi$ fulfills conditions 1. - 3.
\end{lemma}
\begin{proof}
We first show that $\phi_m$ is a Cauchy sequence \dtt wrt to the $L^\infty$-norm and hence its uniform convergence. To this end, let $\ell,m\in\N$. Then
\begin{align*}
\|\phi_{m+\ell}-\phi_m\|_\infty &=\left\|\left[\prod\left(\Id-q^{-(m+1)}D_\frac 1q\right)-\Id\right]\phi_m\right\|_\infty \\
&=\left\| \left[\left(q^{-(m+1)}D_\frac 1q;q^{-1}\right)_\ell-\Id \right]\phi_m \right\|_\infty.
\end{align*}
Utilizing Corollary \ref{qbin}, we obtain
\begin{align*}
\|\phi_{m+\ell}-\phi_m\|_\infty &=\left\| \left[\sum_{k=0}^\ell(-1)^kq^{-\binom k2}\cdot\binom\ell k_\frac 1q\cdot q^{-k(m+1)}D_{q^{-k}}-\Id \right]\phi_m \right\|_\infty \\
&\le \left[\sum_{k=0}^\ell q^{-\left[\binom k2+k(m+1)\right]}\cdot\binom\ell k_\frac 1q-1\right]\cdot\|\phi_m\|_\infty \\
&=\|\phi_m\|_\infty\cdot\sum_{k=1}^\ell q^{-\left[\binom k2+k(m+1)\right]}\cdot\prod_{\nu=0}^{k-1}\frac{1-q^{\ell-\nu}}{1-q^{-(\nu+1)}} \\
&\le\|\phi_m\|_\infty\cdot\sum_{k=1}^\ell q^{-\left[\binom k2+k(m+1)\right]}\cdot\prod_{\nu=0}^{\infty}\frac1{1-q^{-(\nu+1)}} \\
&=\|\phi_m\|_\infty\cdot\sum_{k=1}^\ell q^{-\left[\binom k2+k(m+1)\right]}\cdot\frac 1{\p\left(\frac 1q\right)}.
\end{align*}
Since $\sum_{k=1}^\ell q^{-\left[\binom k2+k(m+1)\right]}\sim q^{-(m+1)}$ for $k\to\infty$ and $\p(a)\ne 0$ for all $a\in(0,1)$, it only remains to show that $\phi_m$ has a uniform upper bound. To this end, we observe that for $m\in\N$ we have
\begin{align*}
\left\|\phi_{m+1}\right\|_{\infty} &= \left\|\left(\mathrm{Id}-q^{-(m+1)}D_{\frac 1q}\right)\phi_{m}\right\|_{\infty} \\
&\le \left\|\phi_{m}\right\|_{\infty}+q^{-(m+1)}\left\|D_\frac 1q\phi_m\right\|_{\infty} \\
&\le \left(1+q^{-(m+1)}\right)\cdot \left\|\phi_{m}\right\|_{\infty}.
\end{align*}
Hence, we can inductively show that
\begin{align*}
\left\|\phi_{m+1}\right\|_{\infty} &\le \prod_{k=0}^{m}\left(1+q^{-(k+1)}\right)\cdot \left\|\phi_{0}\right\|_{\infty} \\
&\le \left\|\phi_{0}\right\|_{\infty}\cdot \exp\left(\sum_{k=0}^{m}\log(1+q^{-(k+1)})\right)
\end{align*}
We hence only need to prove that the series $\sum_{k=0}^\infty\log\left(1+q^{-(k+1)}\right)$ converges. This can be achieved by utilizing the estimate
\[0<\log(1+x)<x\quad\text{for all }x>0.\]
We obtain
\[\sum_{k=0}^\infty\log\left(1+q^{-(k+1)}\right)<\sum_{k=0}^\infty q^{-(k+1)}=\frac q{1-\frac 1q}<\infty.\]
We now proceed by proving the properties 1. - 3. for $\psi$.
\begin{enumerate}
\item Since $\phi_{0}$ is even and the constructive function sequence consists of linear combinations of dilates of $\phi_{0}$, $\psi$ is even, as well.
\item As $\phi_{0}\big|_{\left[-\e,\e\right]}\equiv 1$, we only have to show that $\psi(0)\ne 0$. We can represent the limit function as
\[\psi=\prod_{k=0}^{\infty}\left(\Id-q^{-(k+1)}D_{\frac 1q}\right)\phi_{0}.\]
Due to $\phi_{0}(0)=1$, we obtain that
\[\psi(0)=\prod_{k=0}^{\infty}\left(1-q^{-(k+1)}\right)=\exp\left(\sum_{k=0}^{\infty}\log\left(1-q^{-(k+1)}\right)\right).\]
Due to the concavity of the logarithm, we obtain that
\[\log (1-x)\ge q\log\left(1-\frac 1q\right)\cdot x\quad\text{for all }x\in\left[0,\frac 1q\right].\]
Hence, we obtain that
\[\sum_{k=0}^{\infty}\log\left(1-q^{-(k+1)}\right)\ge q\log\left(1-\frac 1q\right)\cdot\sum_{k=0}^{\infty}q^{-(k+1)}=\frac {\log\left(1-\frac 1q\right)}{1-\frac 1q}.\]
Due to the monotonicity of the exponential function, we then obtain
\[\psi(0)=\exp\left(\sum_{k=0}^{\infty}\log\left(1-q^{-(k+1)}\right)\right)\ge \exp\left(\frac {\log\left(1-\frac 1q\right)}{1-\frac 1q}\right)=\left(\frac {q-1}q\right)^{\frac q{q-1}}>0,\]
since $q>1.$
\item As shown in Lemma \ref{vanmom}, $\phi_{m}$ has $m$ vanishing moments for all $m\in\N$. Hence, $\psi$ has infinitely many vanishing moments. So it remains to prove that $\psi\in\S(\R)$. To this end, we define \[c_{k,\ell,m}:=\|x^{k}\phi_{m}^{(\ell)}(x)\|_{\infty}.\] In order to prove that $\psi\in\S(\R)$, we will show that uniform upper bounds in $m$ exist for the $c_{k,\ell,m}$. \dt Ie, for all $k,\ell\in\N$ we determine a $c_{k,\ell}>0$ such that
\[c_{k,\ell,m}\le c_{k,\ell}\text{ for all }m\in\N.\]
For this purpose we estimate $c_{k,\ell,m+1}$ in terms of $c_{k,\ell,m}$.
\[x^{k}\cdot\phi_{m+1}^{(\ell)}(x)=x^{k}\cdot\partial_{x}^{\ell}\left(\mathrm{Id}-q^{-(m+1)}D_{\frac 1q}\right)\phi_{m}(x)=x^{k}\cdot\left(\mathrm{Id}-q^{-(m+\ell+1)}D_{\frac 1q}\right)\phi_{m}^{(\ell)}(x).\]
Hence, we can estimate
\begin{align*}
c_{k,\ell,m+1} & = \left\|x^{k}\phi_{m+1}^{(\ell)}(x)\right\|_{\infty} \\
&= \left\|x^{k}\cdot\left(\mathrm{Id}-q^{-(m+\ell+1)}D_{\frac 1q}\right)\phi_{m}^{(\ell)}(x)\right\|_{\infty} \\
&\le \left\|x^{k}\phi_{m}^{(\ell)}(x)\right\|_{\infty}+q^{-(m+\ell-k+1)}\left\|\left(\frac xq\right)^{k}\phi_{m}^{(\ell)}\left(\frac xq\right)\right\|_{\infty} \\
&\le \left(1+q^{-(m+\ell-k+1)}\right)\cdot c_{k,\ell,m} \\
&\le \prod_{\nu=0}^{m}\left(1+q^{-(\nu+\ell-k+1)}\right)\cdot c_{k,\ell,0} \\
&\le c_{k,\ell,0}\cdot \prod_{\nu=0}^{m}(1+q^{k-\ell-1-\nu}) \\
&\le c_{k,\ell,0}\cdot \exp\left(\sum_{\nu=0}^{m}\log(1+q^{k-\ell-1-\nu})\right).
\end{align*}
Hence, it is sufficient to prove that the series $\sum_{\nu=0}^{\infty}\log(1+q^{k-\ell-1-\nu})$ converges. To this end, we exploit the following estimate of the logarithm:
\[0<\log(1+x) \le x \quad \text{for all }x>0.\]
With this we obtain
\[\sum_{\nu=0}^{\infty}\log\left(1+q^{k-\ell-1-\nu}\right)\le\sum_{\nu=0}^{\infty}q^{k-\ell-1-\nu}=\frac {q^{k-\ell}}{q-1}.\]
Thus, $c_{k,\ell,m}\le \mathrm{e}^{\frac {q^{k-\ell}}{q-1}}\cdot c_{k,\ell,0}=:c_{k,\ell}$ for all $m\in\N.$ Since $\phi_{0}\in\S(\R)$ by prerequisite, $c_{k,\ell,0}$ is finite for all $k,\ell\in\N$.

\end{enumerate}
\end{proof}

The next lemma gives an explicit representation of $\psi$ as a series of dilates of $\phi_{0}$ by using the $q$-Pochhammer symbol.

\begin{lemma}\label{sum}
Let $\phi_{0}\in\S(\R)$ and let $q>1$. Then
\[\psi=\lim_{m\to\infty}\phi_{m}=\sum_{\ell=0}^{\infty}\frac 1{(q;q)_\ell}\cdot D_{q^{-\ell}}\phi_{0}.\]
\end{lemma}
\begin{proof}
We can rewrite the function $\psi$ as
\[\psi=\prod_{m=0}^\infty\left(\Id-q^{-(m+1)}\cdot D_\frac 1q\right)\phi_0=\left(\frac 1q\cdot D_{\frac 1q};\frac 1q\right)_\infty\phi_0.\]
Due to a result of \cite[Chapter 16]{Eu48}, the following series expansion holds:
\[(z;p)_\infty=\sum_{n=0}^\infty\frac{(-1)^n\cdot p^{n(n-1)/2}}{(p;p)_n}\cdot z^n\]
for $|p|<1$ and for all $z\in\C.$ Since the dilation operator commutes with the multiplication with constants, we can rewrite the limit function $\psi$ as
\begin{align*}
\psi &= \left(\frac 1q\cdot D_{\frac 1q};\frac 1q\right)_\infty\hskip-2mm\phi_0 \\
& = \sum_{\ell=0}^\infty\frac{(-1)^\ell\cdot q^{-\ell(\ell-1)/2}}{(q^{-1};q^{-1})_\ell}\cdot q^{-\ell}\cdot D_{q^{-\ell}}\phi_0 \\
& = \sum_{\ell=0}^\infty\frac{(-1)^\ell}{q^{(\ell+1)\ell/2}\cdot \prod_{k=1}^\ell (1-q^{-k})}\cdot D_{q^{-\ell}}\phi_0 \\
& = \sum_{\ell=0}^\infty\frac{1}{(q;q)_\ell}\cdot D_{q^{-\ell}}\phi_0
\end{align*}

\end{proof}

The next theorem utilizes all of the previous lemmata to show an explicit formula for the limit function $\psi\in\S^*(\R).$

\begin{theorem}\label{construction}
Let $q>1$, $\e>0$ and let $\phi_0\in C_c^\infty(\R)$ be of the form
\[\phi_0(x)=\begin{cases}
1 & \text{for }|x|\le \e, \\
\eta(|x|) & \text{for }|x|\in\big(\e,\e q\big] \\
0 & \text{for } |x|>\e q.
\end{cases},\]
where $\eta\in C^\infty\left(\left[\e,q\e\right]\right)$ is chosen so that $\phi_0\in C^\infty(\R)$. Then, $\psi=\prod_{m=0}^\infty \left(\Id-q^{-(m+1)}D_{\frac 1q}\right)\phi_0$ has the explicit representation
\[\psi(x)=\p\left(\frac 1q\right)+\sum_{\ell=0}^\infty \left[\frac 1{(q;q)_\ell}\cdot\eta\left(\frac{|x|}{q^\ell}\right)-\sum_{k=0}^\ell \frac 1{(q;q)_k}\right]\cdot\1_{\left(\e q^{\ell};\e q^{\ell+1}\right]}(|x|),\]
where $\p$ is the Euler function.
\end{theorem}
\begin{proof}
Due to the form of $\phi_0$ and Lemma \ref{sum} we obtain that 
\begin{align*}
\psi(x) &= \sum_{\ell=0}^\infty \frac 1{(q;q)_\ell}\cdot D_{q^{-\ell}}\phi_0(x) \\
&= \sum_{\ell=0}^\infty \frac 1{(q;q)_\ell}\cdot \left[\1_{\left[-\e q^{\ell},\e q^{\ell}\right]}(x)+\left(D_{q^{-\ell}}\eta\right)(|x|)\cdot\1_{\left(\e q^{\ell},\e q^{\ell+1}\right]}(|x|)\right] \\
&= \sum_{\ell=0}^\infty \frac 1{(q;q)_\ell}\cdot \left(\sum_{k=0}^{\ell-1}\1_{\left(\e q^{k},\e q^{k+1}\right]}(|x|)+\1_{\left[-\e ,\e \right]}(x)+\eta\left(q^{-\ell}\cdot |x|\right)\cdot\1_{\left(\e q^{\ell},\e q^{\ell+1}\right]}(|x|)\right) \\
&= \sum_{\ell=0}^\infty \frac 1{(q;q)_\ell}\cdot \1_{\left[-\e ,\e \right]}(x) + \sum_{k=0}^\infty\1_{\left(\e q^{k},\e q^{k+1}\right]}(|x|)\cdot\sum_{\ell=k+1}^\infty \frac 1{(q;q)_\ell} \\
&\quad +\sum_{\ell=0}^\infty\frac 1{(q;q)_\ell}\cdot\eta\left(q^{-\ell}\cdot |x|\right)\cdot\1_{\left(\e q^{\ell},\e q^{\ell+1}\right]}(|x|).
\end{align*}
By inserting $x=0$ into the upper equation, we obtain that
\[\psi(0)=\sum_{\ell=0}^\infty \frac 1{(q;q)_\ell}.\]
By repeating this for the equation 
\[\psi(x)=\left[\prod_{m=0}^\infty\left(\Id-q^{-(m+1)}\cdot D_{\frac 1q}\right)\phi_0\right](x),\]
we can conclude with $\phi_0(0)=1$ and $(D_af)(0)=(\Id f)(0)$ for all $a>0$ and all $f\in C(\R)$ that 
\begin{align*}
\sum_{\ell=0}^\infty \frac 1{(q;q)_\ell} & =\psi(0)=\left[\prod_{m=0}^\infty\left(\Id-q^{-(m+1)}\cdot \Id\right)\phi_0\right](0)=\prod_{m=0}^\infty\left(1-q^{-(m+1)}\right) \\
& =\prod_{m=1}^\infty\left(1-q^{-m}\right)=\p\left(\frac1q\right).
\end{align*}
This equation together with $\sum_{k=\ell+1}^\infty\frac 1{(q;q)_k}=\p\left(\frac 1q\right)-\sum_{k=0}^\ell\frac 1{(q;q)_k}$ delivers
\[\psi(x)=\p\left(\frac 1q\right)+\sum_{\ell=0}^\infty \left[\frac 1{(q;q)_\ell}\cdot\eta\left(\frac{|x|}{q^\ell}\right)-\sum_{k=0}^\ell \frac 1{(q;q)_k}\right]\cdot\1_{\left(\e q^{\ell};\e q^{\ell+1}\right]}(|x|).\]

\end{proof}
We can also utilize that $\p(x)$ is known for special values of $x$. For instance, we know that
\[\p\left(e^{-\pi}\right)=\frac{e^{\frac\pi{24}}\cdot \Gamma\left(\frac 14\right)}{2^{\frac78}\cdot\pi^{\frac 34}}\]
due to \cite[p. 326]{Be05}. With the choice $q=e^\pi$, we hence obtain the following formula for a function in $\S_1^*(\R)$:
\[\psi(x)=\frac{e^{\frac\pi{24}}\cdot \Gamma\left(\frac 14\right)}{2^{\frac78}\cdot\pi^{\frac 34}}+\sum_{\ell=0}^\infty \left[\frac 1{(e^\pi;e^\pi)_\ell}\cdot \eta(e^{-\ell\pi}|x|)-\sum_{k=0}^\ell \frac 1{(e^{\pi};e^{\pi})_k}\right]\cdot\1_{\left(\e e^{\ell\pi};\e e^{(\ell+1)\pi}\right]}(|x|).\]

\section{Numerical treatment}

In this section we will give a numerical analysis of the computations involved in the explicit formula presented in \Cref{construction} in the previous section.

For a numerical computation of the Euler function, we can utilize Euler's pentagonal number theorem \cite[equation (3.1)]{Er12}, which states that
\[\phi(q) = 1+\sum_{m=1}^\infty (-1)^m\left(q^{m(3m-1)/2}+q^{m(3m+1)/2}\right)\quad\text{for }0<|q|<1.\]
This series expansion provides an excellent error estimate, since
\[\left|\phi(q)-1-\sum_{m=1}^n (-1)^m\left(q^{m(3m-1)/2}+q^{m(3m+1)/2}\right)\right| \le 2q^{(n+1)(3n+2)/2}\quad\text{for }0<|q|<1,\]
as we can easily prove.

Now we will adress the evaluation of the function $\psi$ presented in \Cref{construction}. As we are mainly interested in a function with vanishing moments, we will approximate $\psi$ by a function
\[\phi_n=\prod_{m=0}^{n-1}\left(\Id-q^{-(m+1)}D_{\frac 1q}\right)\phi_0,\]
which has $n$ vanishing moments itself. In the next lemma we prove an approximation error for $\psi$. 

\begin{lemma}\label{error}
Let $q\ge 2$ and $\phi_0$ be given as in \Cref{construction} and let \[\phi_n=\prod_{m=0}^{n-1}\left(\Id-q^{-(m+1)}D_{\frac 1q}\right)\phi_0. \] Then we have
\[\|\psi-\phi_n\|_\infty\le 5q^{-(n+1)}.\]
\end{lemma}
\begin{proof}
We first look for an appropriate representation of $\phi_n$. To this end, we will use Corollary \ref{qbin}. We consider
\begin{align*}
\phi_n &= \prod_{m=0}^{n-1}\left(\Id-q^{-(m+1)}D_{\frac 1q}\right)\phi_0 \\
&= (q^{-1}\cdot D_{\frac 1q};q^{-1})_n\phi_0 \\
&= \sum_{k=0}^n (-1)^kq^{-\binom{k+1}{2}}\cdot\binom nk_{\frac 1q}\cdot D_{q^{-k}}\phi_0\qquad\qquad\qquad (\text{using Corollary }\ref{qbin}) \\
&= \sum_{k=0}^n (-1)^kq^{-\binom{k+1}{2}}\cdot\prod_{\ell=0}^{k-1}\frac{1-q^{\ell-n}}{1-q^{-(\ell+1)}}\cdot D_{q^{-k}}\phi_0 \\
&= \sum_{k=0}^n \frac {(-1)^k}{(q;q)_k}\cdot\prod_{\ell=0}^{k-1}(1-q^{\ell-n})\cdot D_{q^{-k}}\phi_0 \\
&= \sum_{k=0}^n \frac {(-1)^k}{(q;q)_k}\cdot(q^{-n};q)_k\cdot D_{q^{-k}}\phi_0
\end{align*}
Hence, we can write the $L^\infty$-error of the approximation as
\begin{align*}
\|\psi-\phi_n\|_\infty &= \left\| \sum_{k=0}^{n}(-1)^k\cdot \frac {1-(q^{-n};q)_k}{(q;q)_k} D_{q^{-k}}\phi_0 + \sum_{k=n+1}^\infty \frac {(-1)^k}{(q;q)_k} D_{q^{-k}}\phi_0\right\|_\infty \\
&\le \underbrace{\sum_{k=1}^{n} \left| \frac {1-(q^{-n};q)_k}{(q;q)_k}\right|}_{=:\text{(I)}} + \underbrace{\sum_{k=n+1}^\infty \left|\frac 1{(q;q)_k}\right|}_{=:\text{(II)}}.
\end{align*}
For the estimation of (I), we utilize Lemma \ref{qbin0} to obtain
\begin{equation}\label{cell}
1-(q^{-n};q)_k = -\sum_{\ell=1}^k (-1)^{\ell} q^{-n\ell}\cdot q^{\binom\ell 2}\cdot\binom k\ell_q=-\sum_{\ell=1}^k(-1)^\ell c_\ell.
\end{equation}
We will now show that under the given conditions, $\{c_\ell\}_\ell$ is a monotonically decreasing sequence. We observe that
\begin{align*}
c_\ell-c_{\ell+1} &= q^{-n\ell}\cdot q^{\binom\ell 2}\cdot\binom k\ell_q - q^{-n(\ell+1)}\cdot q^{\binom{\ell+1} 2}\cdot\binom k{\ell+1}_q \\
&= q^{-n(\ell+1)}\cdot q^{\binom{\ell} 2}\cdot \binom k\ell_q\cdot \left(q^n -q^\ell\cdot \frac{q^{k-\ell}-1}{q^{\ell+1}-1}\right) \\
&= \frac{q^{-n(\ell+1)}}{q^{\ell+1}-1}\cdot q^{\binom{\ell} 2}\cdot \binom k\ell_q\cdot \left(q^n\cdot (q^{\ell+1}-1) -q^\ell \cdot(q^{k-\ell}-1)\right) \\
&= \frac{q^{-n(\ell+1)}}{q^{\ell+1}-1}\cdot q^{\binom{\ell+1} 2}\cdot \binom k\ell_q\cdot \left(q^{n+1}-q^{n-\ell}-q^{k-\ell}+1\right).
\end{align*}
The last factor is positive, if $q\ge2$, $\ell\ge0$ and $k\le n$. The latter is the case, since (I) covers only the cases where $1\le k\le n$.
Since $c_\ell$ is a positive, monotonically decreasing sequence, the sum in \cref{cell} is an alternating sum over a decreasing sequence. Hence, we obtain the estimate
\[|1-(q^{-n};q)_k| \le c_1 = q^{-n}\cdot\binom k1_q=q^{-n}\cdot \frac{q^k-1}{q-1}.\]
This delivers the following estimate for (I):
\begin{align*}
(\text{I}) &= \sum_{k=1}^n\left|\frac {1-(q^{-n};q)_k}{(q;q)_k}\right| \\
&\le \frac 1{q^{n}(q-1)}\cdot\sum_{k=1}^n (q^k-1)\cdot\frac 1{|(q;q)_k|} \\
&= \frac 1{q^{n}(q-1)}\cdot\sum_{k=1}^n\frac 1{|(q;q)_{k-1}|}.
\end{align*}
By adding up (II) and the estimate for (I), we obtain
\begin{align*}
\|\psi-\phi_n\| &\le \text{(I)+(II)} \\
&\le \frac 1{q^{n}(q-1)}\cdot\sum_{k=1}^n\frac 1{|(q;q)_{k-1}|} + \sum_{k=n+1}^\infty\frac 1{|(q;q)_{k}|} \\
&\le \frac 1{q^{n}(q-1)}\cdot \sum_{k=0}^{n-1}\frac 1{|(q;q)_{k}|} + \frac 1{q^{n+1}-1}\cdot\sum_{k=n}^{\infty} \frac 1{|(q;q)_{k}|} \\
&\le \frac 1{q^{n}(q-1)}\cdot \sum_{k=0}^{\infty}\frac 1{|(q;q)_{k}|}  \\
&\le 5\cdot q^{-(n+1)}.
\end{align*}
The last estimate can be obtained via the condition $q\ge 2$ by the considerations
\[\frac 1{q-1}\le 2q^{-1}\quad\text{and}\quad\sum_{k=0}^{\infty}\frac 1{|(q;q)_{k}|} = \sum_{k=0}^\infty\prod_{\ell=1}^k \frac1{q^\ell-1} \le 1+\underbrace{\frac 1{q-1}}_{\le 1}\cdot\sum_{m=0}^\infty\underbrace{\frac 1{(q^2-1)^m}}_{\le 3^{-m}}\le\frac 52.\]

\end{proof}

With this result we have an error estimate for a numerical application of \Cref{construction} to the Taylorlet transform. In order to see the effect of the approximation on the decay rate, we briefly revisit \Cref{detect} rsp. \hskip-1mm Corollary \ref{cor}. In both statements we observe the decay rate of the Taylorlet transform $\mathcal{T}f(a,s,t)$ of a feasible function $f(x)=c\cdot[x_1-q(x_2)]^j\cdot\1_{\R_\pm}(x_1-q(x_2))$ with a singularity function $q\in C^\infty(\R)$ for $a\to 0$. Additionally, both revolve around the cases 
\begin{equation}\tag{$A_k$}\label{A_k} 
s_i = q^{(i)}(t)\quad\text{for all }i\in\{0,\ldots,k\} 
\end{equation}
When using a restrictive analyzing Taylorlet as constructed in \Cref{construction}, Corollary \ref{cor} states that $\mathcal{T} f(a,s,t)$ decays superpolynomially fast for $a\to 0$ if and only if $(A_n)$ does not hold. Thus it provides a sharp distinction between $(A_n)$ and $\neg(A_n)$. For a numerical treatment of the Taylorlet transform we are bound to using Taylorlets with finitely vanishing moments and thus \Cref{detect} applies. Despite not providing a superpolynomial versus polynomial decay rate scenario, it still guarantees a significantly lower decay rate of the Taylorlet transform, if $(A_n)$ holds.

For an illustration of the functioning of the Taylorlets, we will show images of the Taylorlet transform of different twodimensional functions.

The Taylorlets we use for the images have 5 vanishing moments of second order in $x_1$-direction and are of the form suggested in \Cref{construction}.
\begin{align*}
\tau(x)&=g(x_1)\cdot h(x_2), \\
h(x_2)&=e^{-x_2^2}, \\
g(x_1)&=\phi_{10}\left(\sqrt{|x_1-\tfrac 18|}\right), \\
\phi_{10}(t)&=\prod_{m=0}^9 (\Id-2^{-(m+1)}D_{\frac 12})\phi_0(t), \\
\phi_0(t)&=\begin{cases}
1, & |t|\le \frac 14, \\
(128t^3-144t^2+48t-4), & |t|\in\left(\frac 14,\frac 12\right), \\
0, & |t|\ge \frac 12.
\end{cases}
\end{align*}

As stated in \Cref{detect}, the restrictiveness is important for the lower bound of the Taylorlet transform and hence for detecting the Taylor coefficients of the singularity function. In the upper definition of $\tau$ the shift in the step $g(x_1)=\phi_{10}\left(\sqrt{|x_1-\tfrac 18|}\right)$ provides the restrictiveness of the Taylorlet while preserving the vanishing moments, as the following lemma show. It must be mentioned, though, that the upper Taylorlet is only in $C^1(\R^2)$, but not $C^\infty(\R^2)$, because $\phi_0$ is not $C^2$ in the points $\pm\tfrac 14$ and $\pm\tfrac 12$.

\begin{lemma}\label{rest}
Let $q>1$ and let $\phi_0$ be of the form given in \Cref{construction} and let
\[\phi_{m+1}=\left(\Id-q^{-(m+1)}D_{\frac 1q}\right)\phi_m\]
for all $m\in\N$. If
\begin{enumerate}[label=\alph*)]
\item $t_0\in\left(-\e^{v_n},\e^{v_n}\right)$, $M\in\N$ and $g(t):=\phi_{Mv_n}\left(\sqrt[v_n]{|t-t_0|}\right)$ for all $t\in\R$,
\item $h\in\S(\R)$ with $\int_\R h(t)dt\ne 0$,
\end{enumerate}
then $\tau=g\otimes h$ has $M$ vanishing moments of order $n$ in $x_1$-direction and is restrictive.
\end{lemma}
\begin{proof}
Let $\tilde\phi(t):=\phi_{Mv_n}\left(\sqrt[v_n]{|t|}\right)$. Due to \Cref{construction} there exists a $c>0$ such that
\begin{equation} \label{c}
\tilde\phi\big|_{\left[-\e^{v_n},\e^{v_n}\right]} \equiv c.
\end{equation}
Hence, with $g=\tilde\phi(\ \cdot-t_0)$ and $t_0\in\left(-\e^{v_n},\e^{v_n}\right)$ we obtain that
\[ g\big|_{\left[t_0-\e,\e-t_0\right]} \equiv c \]
and so $g\in\S(\R)$ still. Furthermore, Proposition \ref{sqrt} and Lemma \ref{vanmom} deliver that $g$ has $M$ vanishing moments of order $n$. For the restrictiveness of $\tau$, we have to show that
\begin{enumerate}[label=(\roman*)]
\item $g(0)\ne 0$ and $\int_0^\infty g(t)t^jdt\ne 0$ for all $j\in\{0,\ldots,r-1\}$ and
\item $\int_\R h(t)dt\ne 0$.
\end{enumerate}
Property (ii) is already given by condition b). Hence, (i) remains to show.

Lemma \ref{vanmom} in addition states that $\int_{\R_\pm}\phi_{Mv_n}(t)t^mdt=0$ for all $\ell\in\{0,\ldots,Mv_n-1\}$. We will now show that a similar property is true for $\tilde\phi$. The variable substitution $t= u^{v_n}$ delivers
\begin{align*}
\int_0^\infty\tilde\phi(t)t^mdt &= \int_0^\infty\phi_{Mv_n}\left(\sqrt[v_n]{|t|}\right)t^mdt \\
&= \int_0^\infty \phi_{Mv_n}(u) u^{mv_n}\cdot v_n\cdot u^{v_n-1} du \\
&=v_n\cdot\int_0^\infty \phi_{MV_n}(u)u^{(m+1)v_n-1} du=0
\end{align*}
for all $m\in\{0,\ldots,M-1\}$ due to Lemma \ref{vanmom}. With this result, we can show property (i). For $m\in\{0,\ldots,M-1\}$, we have
\begin{align*}
\int_0^\infty g(t)t^mdt&=\int_0^\infty \tilde\phi(t-t_0)t^mdt \\
&=\int_{-t_0}^\infty \tilde\phi(u)(u+t_0)^m dt \\
&= \sum_{k=0}^m \binom mk t_0^{m-k}\cdot\int_{-t_0}^\infty \tilde\phi(u)u^k dt \\
&= \sum_{k=0}^m \binom mk t_0^{m-k}\cdot \Bigg(\int_{-t_0}^0\underbrace{\tilde\phi(u)}_{\equiv c\text{ for }|u|\le t_0\quad (\ref{c})}u^kdu+\underbrace{\int_0^\infty \tilde\phi(u)u^k dt}_{=0} \Bigg) \\
&= c\cdot\sum_{k=0}^m \binom mk t_0^{m-k}\int_{-t_0}^0u^k du \\
&= c\cdot\sum_{k=0}^m \binom mk t_0^{m-k}\cdot\frac 1{k+1}\cdot[-(-t_0)^{k+1}] \\
&= -\frac {c\cdot t_0^{m+1}}{m+1}\cdot \sum_{k=0}^m \binom {m+1}{k+1} (-1)^{k+1} \\
&= -\frac {c\cdot t_0^{m+1}}{m+1}\cdot\Bigg( \underbrace{\sum_{k=0}^{m+1} \binom {m+1}{k} (-1)^{k}}_{=0} - \binom{m+1}0 (-1)^0\Bigg) \\
&= \frac{c\cdot t_0^{m+1}}{m+1}\ne 0.
\end{align*}

\end{proof}
The functions we will analyze with the Taylorlet transform are of the form
\[f(x)=\1_{\R_+}\left(x_1-q(x_2)\right),\]
where $q\in C^\infty(\R)$ is the singularity function of $f$. In order to efficiently compute the Taylorlet transform of such functions, we utilize the following lemma allowing a 1-dimensional and thus faster integration.
\begin{lemma}\label{1d}
Let $\tau$ be of the form given in Lemma \ref{rest}, let $f(x)=\1_{\R_+}\left(x_1-q(x_2)\right)$ with $q\in C^\infty(\R)$ and $n\in\N,$ $n\ge 2$. Furthermore let 
\[G(w):=av_n\cdot\int_{\sqrt[v_n]{|w-t_0|}}^\infty\phi_{Mv_n}(u)|u|^{v_n-1}du\]
for $w\in|R$. Furthermore, let $T:=G\otimes h$ and let $T_{ast}(x):=T\left(A_\frac 1a S_{-s}(x-t e_2)\right)$ for $x\in\R^2$.
Then 
\[\T_{\tau}f(a,s,t)=\int_\R T_{ast}\cvec{q(u) \\ u} du\]
for all $a>0$, $s\in\R^{n+1}$, $t\in\R$.
\end{lemma}
\begin{proof}
\dtf Wlog let $t=0$. By rewriting the Taylorlet transform we obtain
\begin{align*}
\mathcal{T}_\tau f(a,s,0) &= \int_{\R^2} \tau\cvec{\frac 1a\cdot \left[x_1-\sum_{k=0}^n \frac{s_kx_2^k}{k!}\right] \\ \frac{x_2}{a^\alpha}} \1_{\R_+}\left(x_1-q(x_2)\right)dx \\
&= \int_\R \int_{q(x_2)}^{\infty} g\left(\frac 1a\cdot \left[x_1-\sum_{k=0}^n \frac{s_kx_2^k}{k!}\right]\right)dx_1 h\left(\frac{x_2}{a^\alpha}\right)dx_2 \\
&= \int_\R \int_{q(x_2)}^{\infty} \phi_{Mv_n}\left(\sqrt[v_n]{\left|\frac 1a\cdot \left[x_1-\sum_{k=0}^n \frac{s_kx_2^k}{k!}\right]-t_0\right|}\right) dx_1 h\left(\frac{x_2}{a^\alpha}\right)dx_2.
\end{align*}
Performing the variable substitution
\[x=\cvec{\sgn(y_1)\cdot\left[a(y_1^{v_n}+t_0)+\sum_{k=0}^n\frac{s_kx_2^k}{k!}\right] \\ x_2},\quad \frac{dx_1}{dy_1}=av_n|y_1|^{v_n-1},\]
delivers
\begin{align*}
\mathcal{T}_\tau f(a,s,0) &= av_n\cdot\int_\R \int_{\sqrt[v_n]{\left|\frac 1a\cdot \left[q(y_2)-\sum_{k=0}^n \frac{s_ky_2^k}{k!}\right]-t_0\right|}}^\infty \phi_{Mv_n}(y_1)|y_1|^{v_n-1} dy_1 h\left(\frac{y_2}{a^\alpha}\right)dy_2 \\
&= \int_\R G\left(\frac 1a\cdot \left[q(y_2)-\sum_{k=0}^n \frac{s_ky_2^k}{k!}\right]\right) h\left(\frac{y_2}{a^\alpha}\right)dy_2 \\
&= \int_\R T_{as0}\cvec{q(y_2) \\ y_2} dy_2.
\end{align*}

\end{proof}

For the implementation of the Taylorlet transform in Matlab, we utilized Lemma \ref{1d} and employed the one-dimensional adaptive Gauss-Kronrod quadrature \texttt{quadgk} in Matlab for the evaluation of the integrals.

\begin{table}

\centering

\begin{tabular}{|c|c|c|}
\hline
$s_{0}$ & $s_{1}$ & $s_{2}$  \\
\hline
\begin{minipage}{.29\textwidth}
\includegraphics[width=\linewidth, trim = 0 -5 0 -5 ]{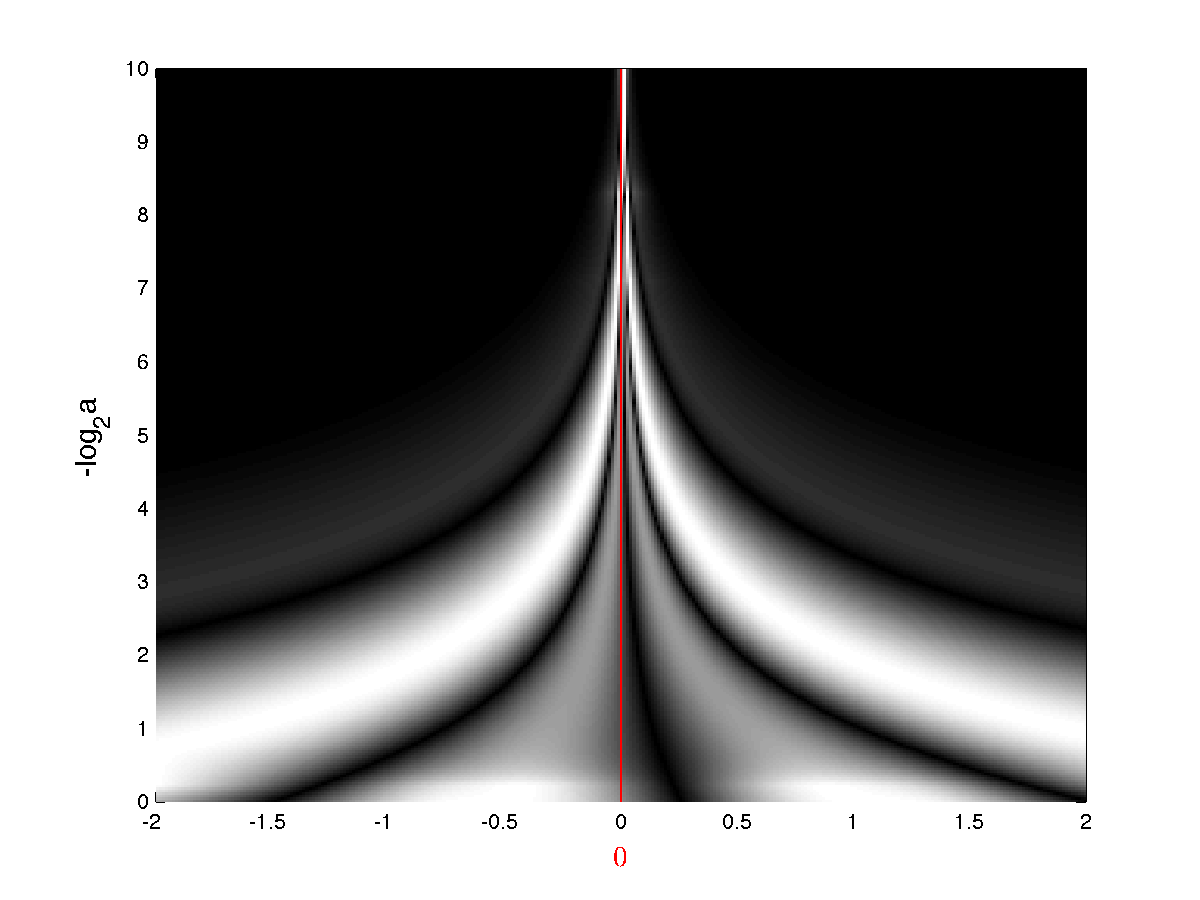}
\end{minipage} & 
\begin{minipage}{.29\textwidth}
\includegraphics[width=\linewidth, trim = 0 -5 0 -5 ]{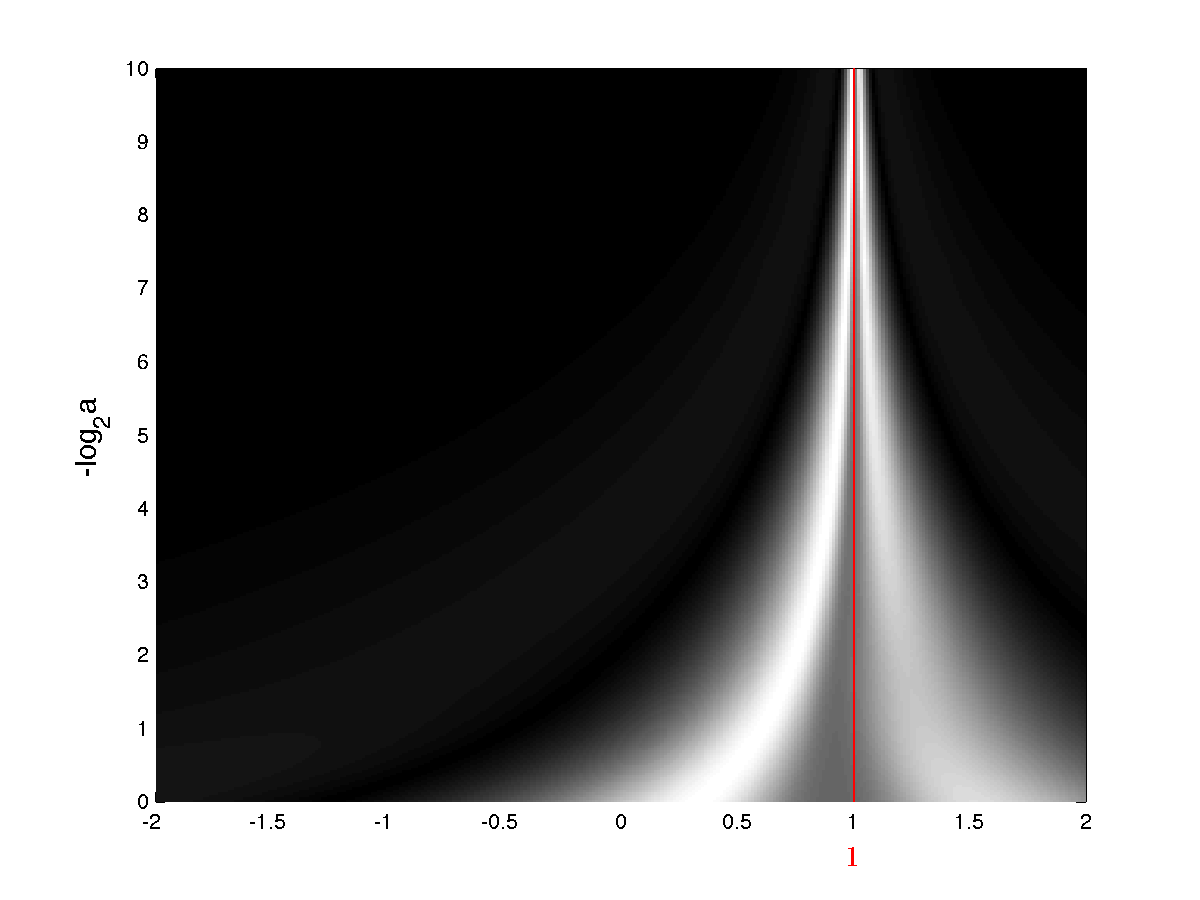}
\end{minipage} & 
\begin{minipage}{.29\textwidth}
\includegraphics[width=\linewidth, trim = 0 -5 0 -5 ]{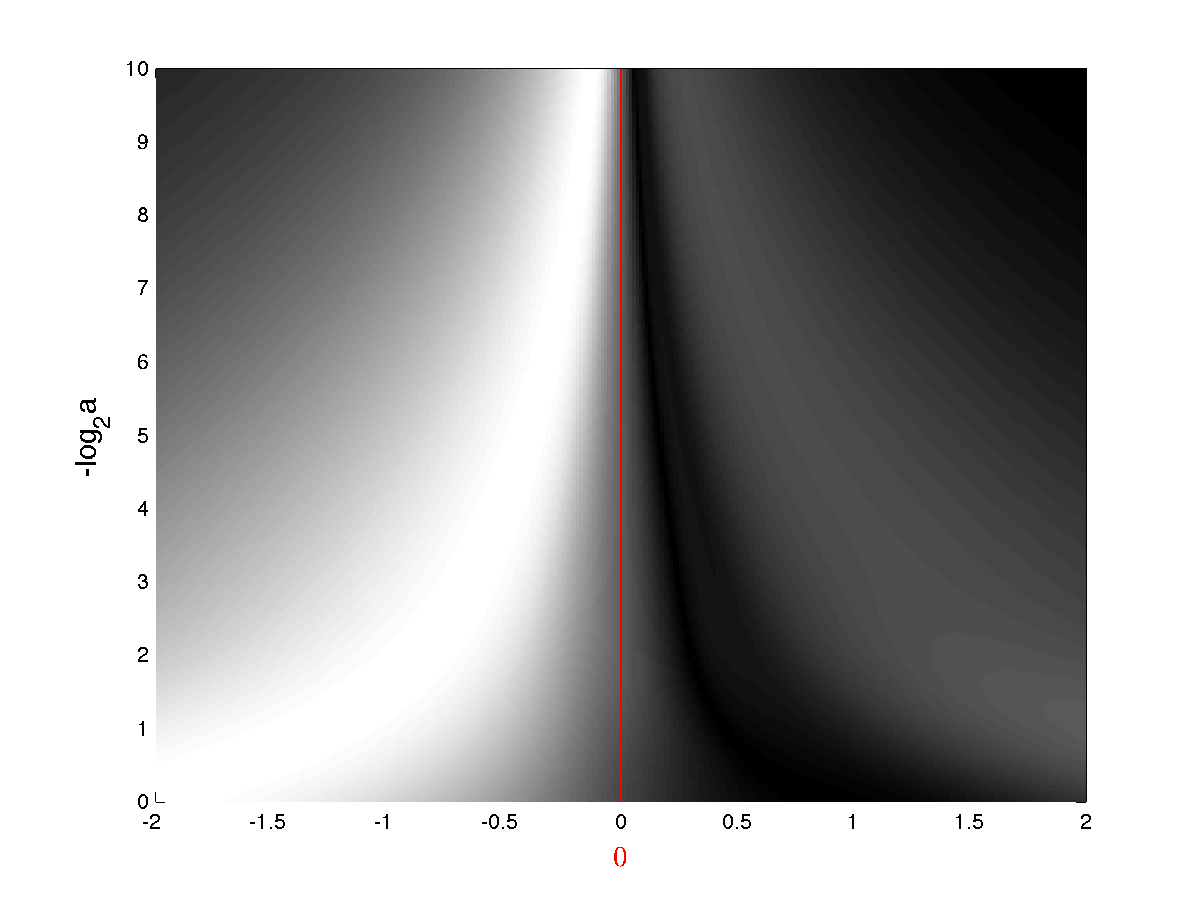}
\end{minipage} \\
\hline
\begin{minipage}{.29\textwidth}
\includegraphics[width=\linewidth, trim = 0 -5 0 -5 ]{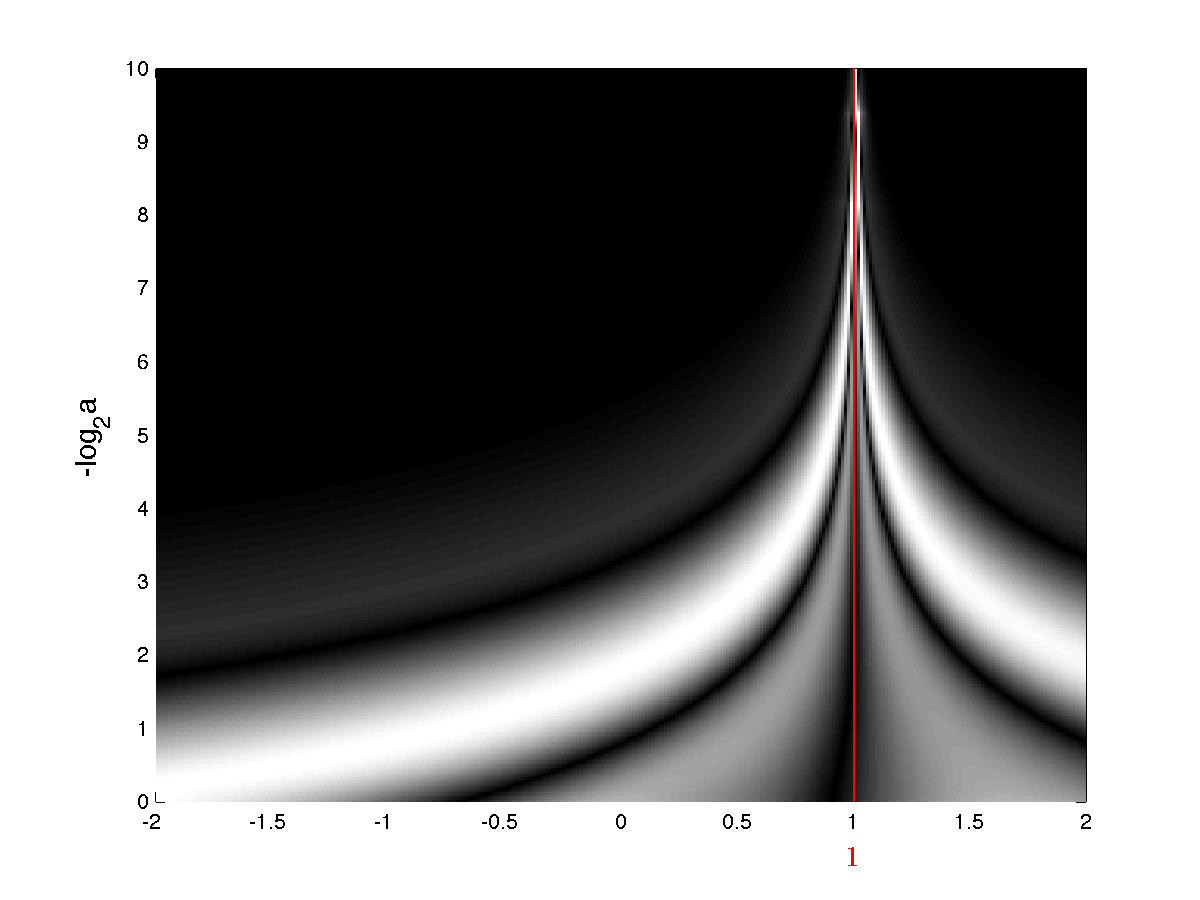}
\end{minipage} & 
\begin{minipage}{.29\textwidth}
\includegraphics[width=\linewidth, trim = 0 -5 0 -5 ]{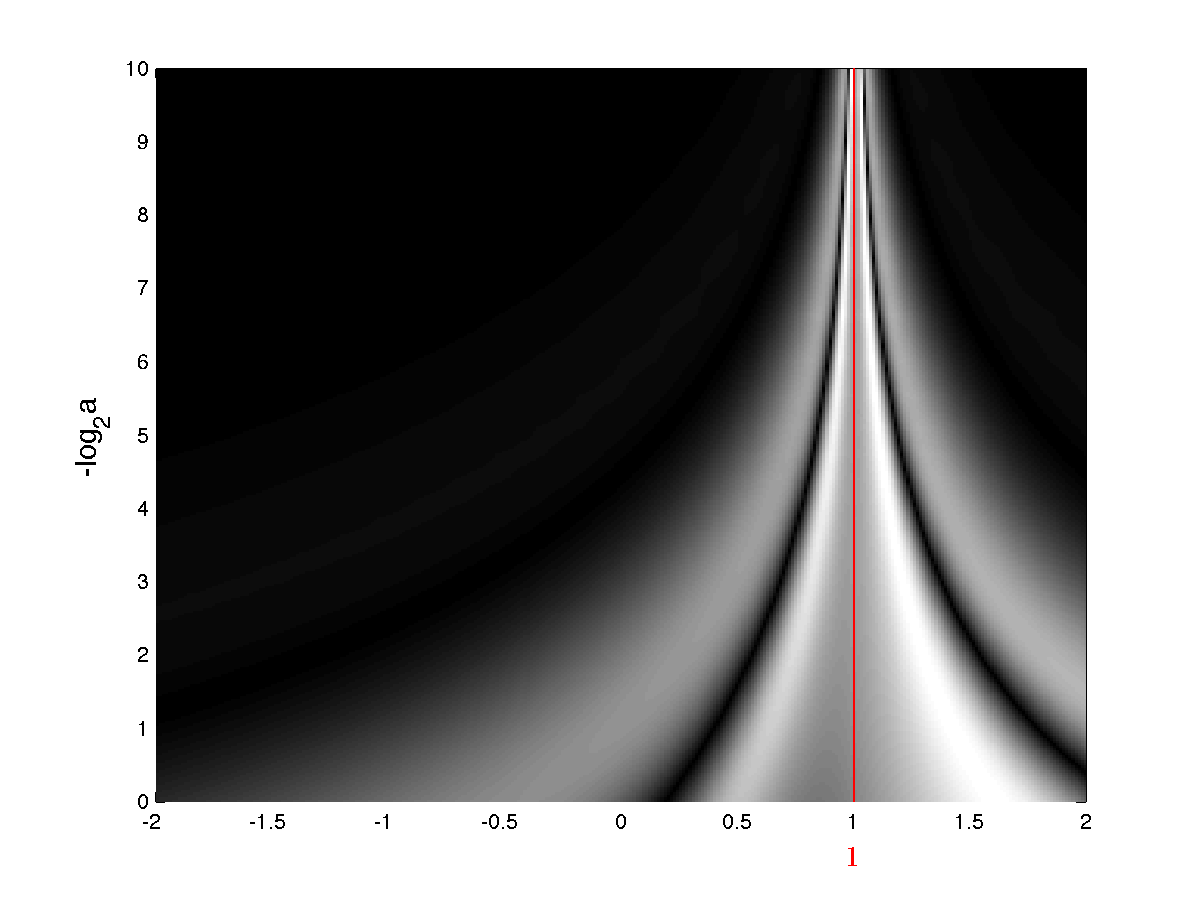}
\end{minipage} & 
\begin{minipage}{.29\textwidth}
\includegraphics[width=\linewidth, trim = 0 -5 0 -5 ]{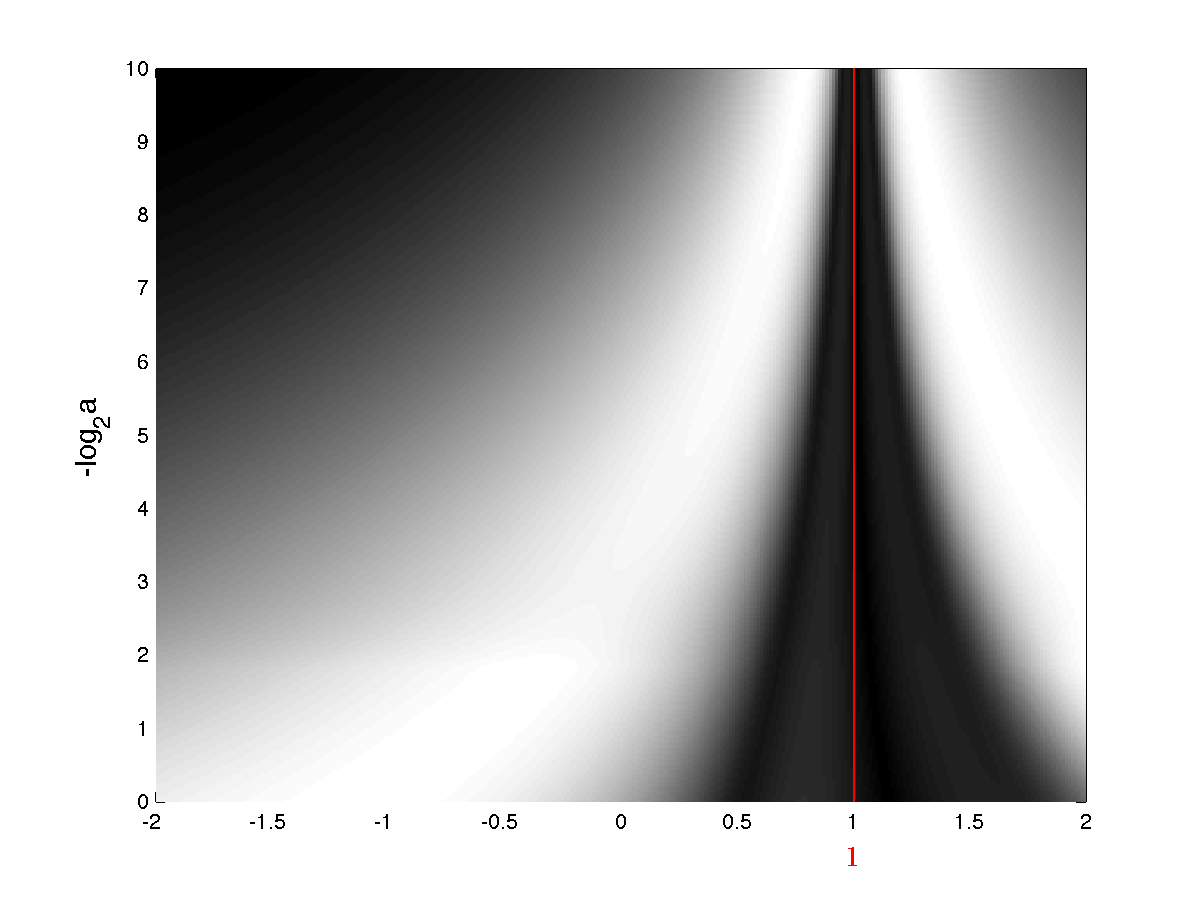}
\end{minipage} \\
\hline
\begin{minipage}{.29\textwidth}
\includegraphics[width=\linewidth, trim = 0 -5 0 -5 ]{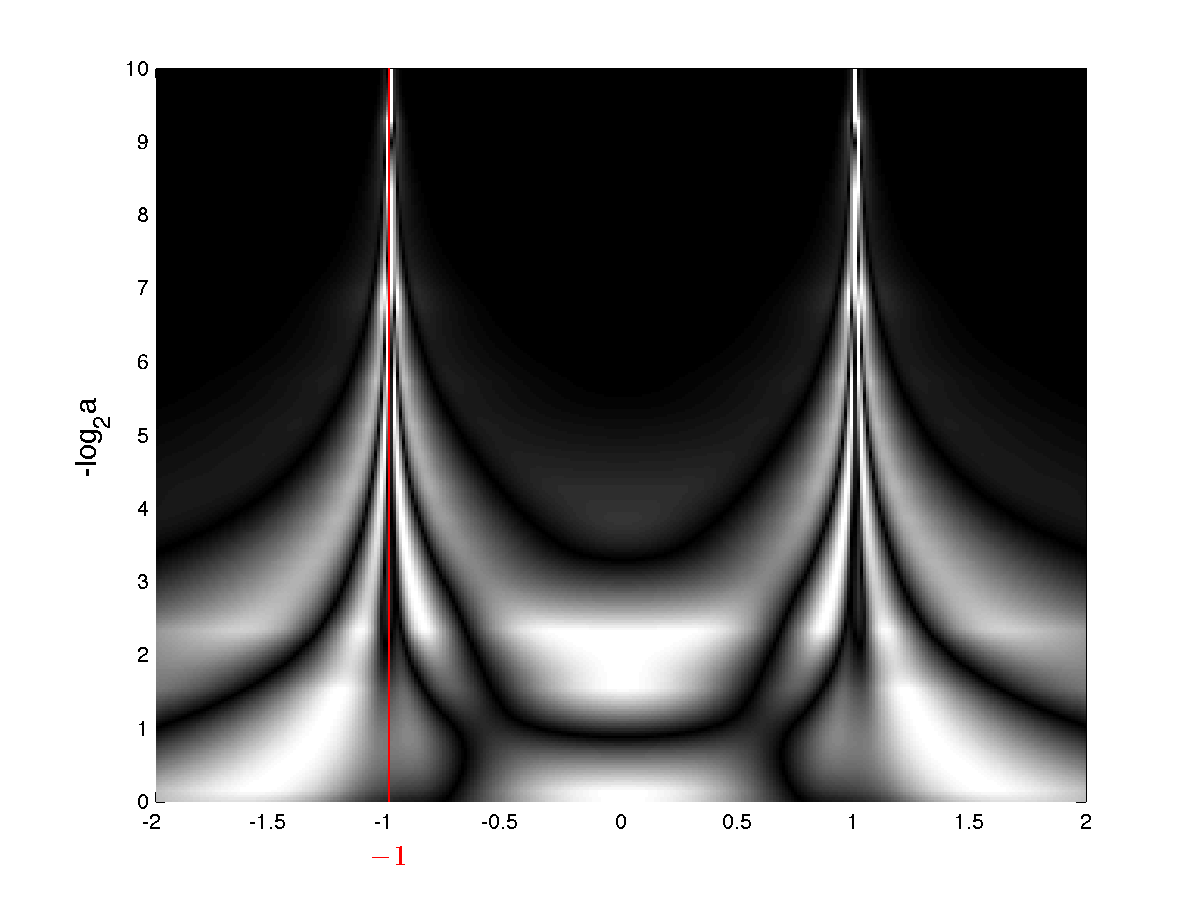}
\end{minipage} & 
\begin{minipage}{.29\textwidth}
\includegraphics[width=\linewidth, trim = 0 -5 0 -5 ]{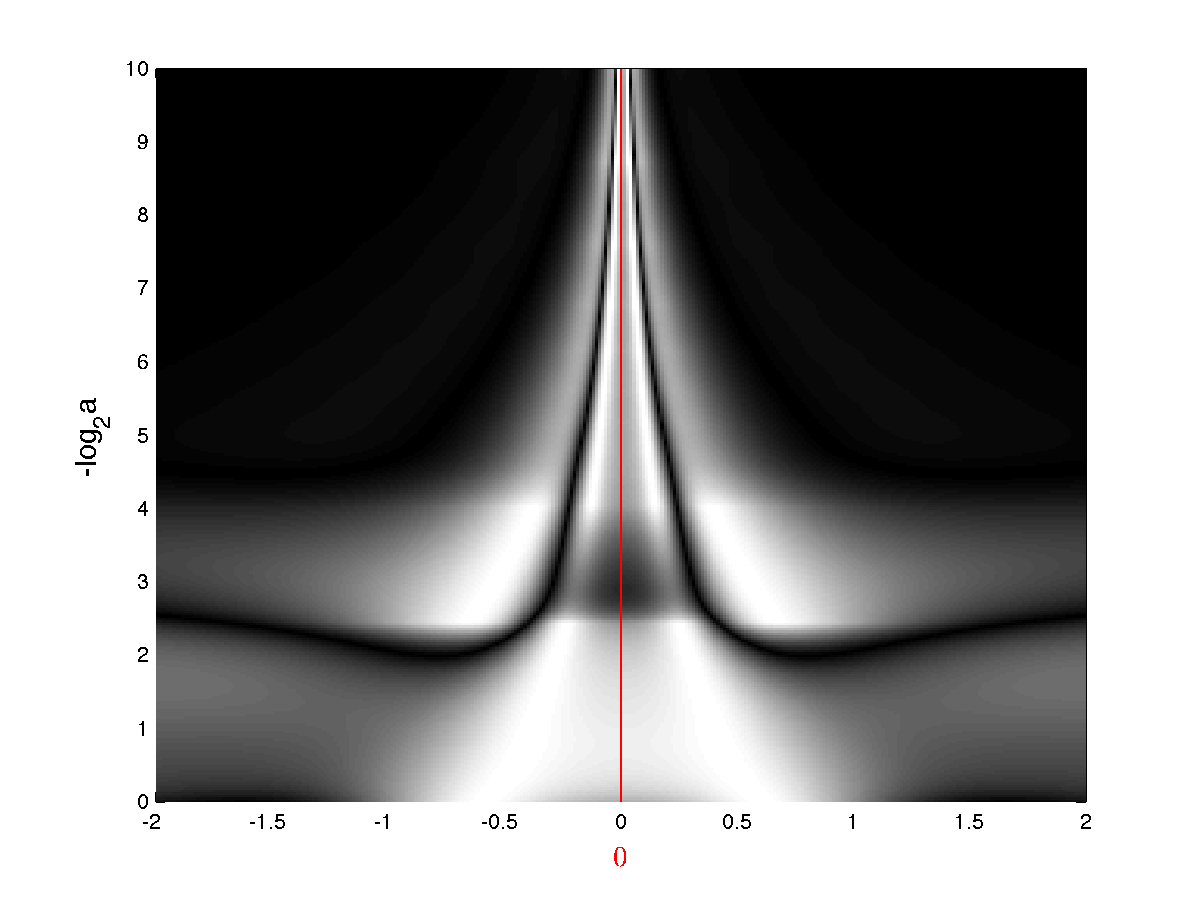}
\end{minipage} & 
\begin{minipage}{.29\textwidth}
\includegraphics[width=\linewidth, trim = 0 -5 0 -5 ]{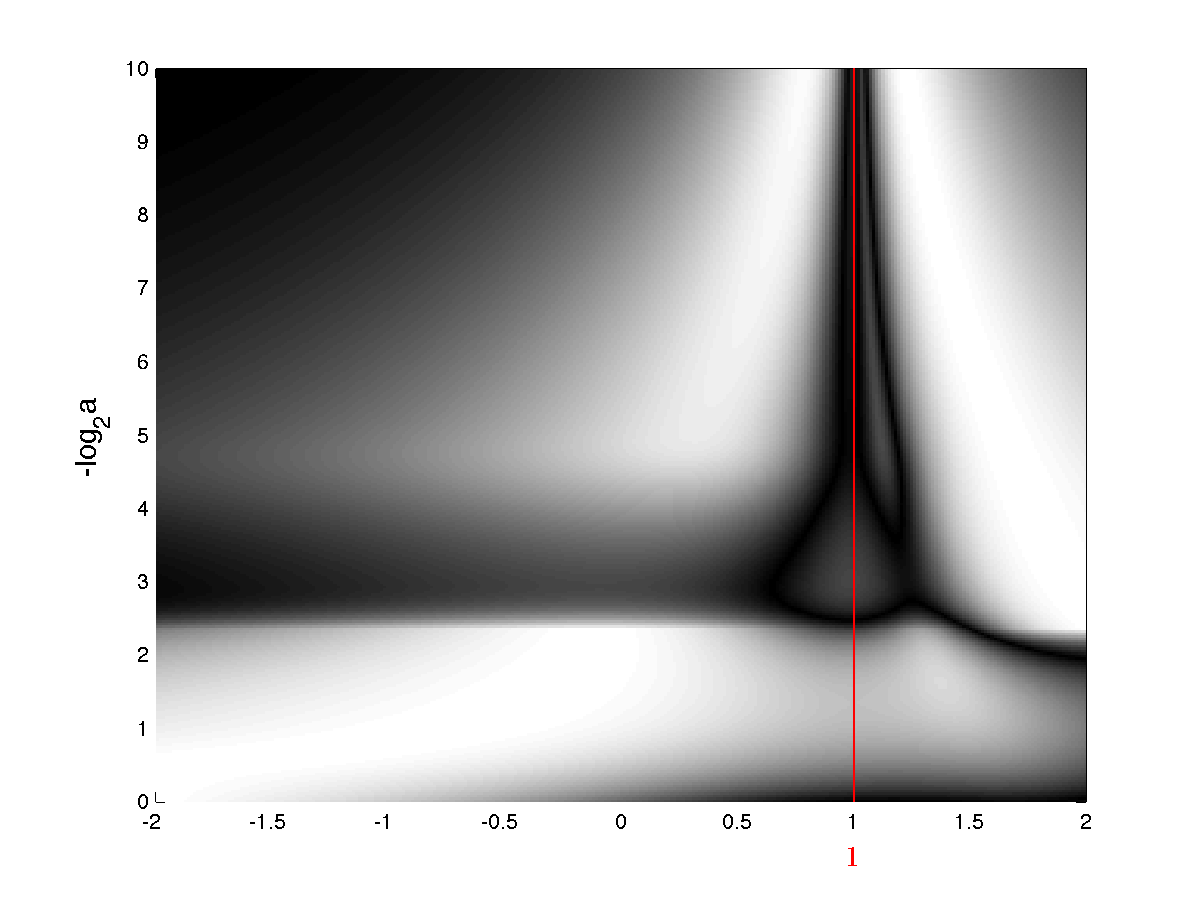}
\end{minipage}\\
\hline
\end{tabular}
\caption{Plots of the Taylorlet transform $\mathcal{T}f(a,s,0)$ for $f_1(x)=\1_{\R_+}(x_1- \sin x_2)$ (upper row), $f_2(x)=\1_{\R_+}\left(x_1-e^{x_2}\right)$ (middle row) and $f_3(x)=\1_{B_1}(x)$ (bottom row). The vertical axis shows the dilation parameter in a logarithmic scale $-\log_2a$. The horizontal axis shows the location $s_0$ (left), the slope $s_1$ (center) and the parabolic shear $s_2$ (right). The respective true value is indicated by the vertical red line. The values of $\alpha$ change with $s_{i}$: for $s_{0}$ we use $\alpha=1.01$, for $s_{1}$ we have $\alpha=0.51$ and during the search for $s_{2}$ we set $\alpha=0.34$. The Taylorlet transform was computed for points $(a,s_i)$ on a $300\times 300-$grid. We can observe the paths of the local maxima \dtt wrt the respective shearing variable as they converge to the correct related geometric value through the scales. Due to the vanishing moment conditions of higher order, the local maxima display a fast convergence to the correct value. The bottom left image shows that the Taylorlet transform of $f_3$ exhibits the two singularities $s_0=\pm1$, since this function is not of the form $\1_{\R_+}(x_1-q(x_2))$.}
\end{table}

Table 4 shows plots of the Taylorlet transform $\mathcal{T}f(a,s,0)$ of the three functions \(f_1(x)=\1_{\R_+}(x_1- \sin x_2)\), \(f_2(x)=\1_{\R_+}\left(x_1-e^{x_2}\right)\) and \(f_3(x)=\1_{B_1}(x)\). The vertical axis describes the dilation parameter in a binary logarithmic scale while the horizontal axis shows location, slope and parabolic shear rsp. 

These plots represent a search for successive Taylor coefficients of the respective singularity functions around the origin. This is possible by exploiting \Cref{detect} which states that the decay rate of the Taylorlet transform $\mathcal{T} f(a,s,t)$ for $a\to 0$ depends on $\alpha$ and on the highest $k\in\{0,\ldots,n\}$ for which the condition $(A_k)$ does not hold. In order to find $q(t)$, we hence compute the Taylorlet transform with $\alpha>1$ first. The choice of $\alpha$ and the restrictiveness of the Taylorlet, as guaranteed by Lemma \ref{rest}, ensure a decay rate of
\[\mathcal{T}f(a,s_{0},t)\sim 1\quad\text{for } a\to 0\]
for $s_{0}=q(t)$ due to \Cref{detect}. 

The following procedure is an adaption of the method of wavelet maximum modulus \cite{MaHw92}. We observe the paths of the local maxima \dtt wrt $s_0$ for decreasing $a$. As in the method of maximum modulus, the local extrema of the Taylorlet transform converge to $s_0=q(t)$ for decreasing $a$. After finding $q(t)$, we fix $s_0=q(t)$, choose an $\alpha\in\left(\frac 12,1\right)$ and repeat the procedure to find $q'(t)$ and to search for $q''(t)$ with an $\alpha\in\left(\frac 13,\frac12\right)$ in the final step.

For a better visibility of the local maxima, we normalized the absolute value of the Taylorlet transform such that the maximal value in each scale is 1. 

\section{Conclusion and Discussion}

In this article we utilized methods from q-calculus in the construction of a function $g\in\S_1^*(\R)$ which additionally is constant around the origin (\Cref{construction}). This allows for the creation of an analyzing Taylorlet of arbitrary order with infinitely many vanishing moments (Proposition \ref{sqrt}). In the explicit formula of the constructed function, the Euler function appears and exhibits an inherent connection between the presented method and the field of combinatorics. In section 5 we presented a numerical analysis of the evaluation of the function constructed in \Cref{construction} which preserves several important properties such as smoothness, decay rate and some, although not all of the vanishing moments. In numerical experiments we illustrated that our mathematical results translate into practice.

In the following consideration let $\tau$ be a restrictive analyzing Taylorlet of order $n$ with infinitely many vanishing moments in $x_1$-direction. As Corollary~\ref{cor} states, such a Taylorlet $\tau$ allows for a precise determination of the first $n+1$ Taylor coefficients of the singularity function $q$ by means of the decay rate of the Taylorlet transform. Interestingly, the first two Taylor coefficients are deeply linked to the concept of the wavefront set which includes localization and directionality of singularities. Furthermore, as was shown by Grohs, the shearlet transform allows for a resolution of the wavefront set if the shearlet is a Schwartz function and exhibits infinitely many vanishing moments in $x_1$-direction \cite[Thm 6.4]{gr11}. Hence, the Taylorlet $\tau$, similarly, is a good starting point for a generalization of the concept of wavefront set which addicionally includes local curvature and higher order geometric information of the distributed singularity. In a similar fashion as in the characterization of the classical wavefront set by the continuous shearlet transform, we can define a generalized wavefront of a tempered distribution $f\in\S'(\R^2)$ by using a Taylorlet $\tau$ as described above and $\alpha\in\left(\frac 1{n+1},\frac 1n\right)$ via
\begin{align*}
\mathcal{WF}_n(f)^c=\big\{&(t,s_0,\ldots,s_n)\in\R^{n+2}:\,\exists\text{ open neighborhood }U\text{ of }(t,s_0,\ldots,s_n):\\
&\mathcal{T}_\tau f(a,\cdot,\cdot)\text{ decays superpolynomially fast for }a\to 0\text{ in }U \text{ globally}\big\}.
\end{align*}
This concept would enable a more precise description and analysis of singularities than the definition of the wavefront set allows. To the knowledge of the authors this idea has not been pursued to this date and so a thorough investigation is still needed, but beyond the scope of this paper.

\section*{Acknowledgements}

The first author expresses his gratitude for the support by the DFG project FO 792/2-1 "Splines of complex order, fractional operators and applications in signal and image processing", awarded to Brigitte Forster. The work of the second author was supported by Portuguese funds through the CIDMA - Center for Research and Development in Mathematics and Applications, and the Portuguese Foundation for Science and Technology (``FCT--Funda\c{c}\~ao para a Ci\^encia e a Tecnologia''),  within project UID/MAT/ 0416/2013 and the sabbatical grant  SFRH/BSAB/135157/2017.

\bibliographystyle{alpha}
\bibliography{euler}

\end{document}